\documentclass[12pt]{article}
\usepackage{fullpage}
\usepackage{amsmath,amssymb,amsthm}
\usepackage{float}
\usepackage{enumerate}
\usepackage{hyperref}

\title{Avoiding Squares and Overlaps Over the Natural Numbers}
\author{
Mathieu Guay-Paquet\thanks{Supported by an NSERC Alexander Graham Bell Canada Graduate Scholarship.} \\
Department of Combinatorics and Optimization \\
\texttt{mguaypaq@math.uwaterloo.ca} \\
\and
Jeffrey Shallit \\
School of Computer Science \\
\texttt{shallit@cs.uwaterloo.ca} \\[1ex]
University of Waterloo \\
Waterloo, ON~~N2L~3G1 \\
Canada}

\newtheorem{theorem}{Theorem}
\newtheorem{lemma}[theorem]{Lemma}

\newtheorem{corollary}[theorem]{Corollary}
\theoremstyle{definition}
\newtheorem{definition}[theorem]{Definition}
\newtheorem{remark}[theorem]{Remark}

\newcommand{\N}{\mathbb{N}}

\newcommand{\w}{\mathbf{w}}

\begin{document}
\maketitle
\begin{abstract}
We consider avoiding squares and overlaps over the natural numbers, using a greedy algorithm that chooses the least possible integer at each step; the word generated is lexicographically least among all such infinite words. In the case of avoiding squares, the word is $01020103\cdots$, the familiar ruler function, and is generated by iterating a uniform morphism. The case of overlaps is more challenging. We give an explicitly-defined morphism $\varphi \colon \N^* \to \N^*$ that generates the lexicographically least infinite overlap-free word by iteration. Furthermore, we show that for all $h,k \in \N$ with $h \leq k$, the word $\varphi^{k-h}(h)$ is the lexicographically least overlap-free word starting with the letter $h$ and ending with the letter $k$, and give some of its symmetry properties.
\end{abstract}

\section{Introduction}

Avoidability problems play a significant role in combinatorics on words. Typically we are given a finite alphabet $\Sigma$, and we want to know if there exist infinite words over $\Sigma$ that avoid various patterns, such as squares and overlaps. A \emph{square} is a nonempty word of the form $xx$, such as the French word \texttt{chercher}. An \emph{overlap} is a word of the form $axaxa$ where $a$ is a single letter and $x$ is a (possibly empty) word, such as the French word \texttt{entente}. An overlap is sometimes called a $(2^+)$-power, because it is just slightly more than a square. In two famous papers, the Norwegian mathematician Axel Thue \cite{Thue:1906,Thue:1912,Berstel:1995} proved that there exist infinite binary words containing no overlaps, and infinite words over a 3-letter alphabet containing no squares.

Suppose we try to generate an infinite squarefree word over the alphabet $\Sigma_3 = \{0, 1, 2\}$ letter by letter, using the familiar backtracking algorithm \cite{Golomb&Baumert:1965}. At every step, we choose the smallest letter possible that maintains the property of not having a square; if no such letter exists, we are forced to backtrack to a previous letter and increment it. For example, this approach generates the string $w = 0102010$, at which point no letter in $\Sigma_3$ can be appended without getting a square. Thus we are forced to backtrack one letter, replacing the last letter of $w$ with $2$ to obtain $0102012$, and we continue from there. Although this approach will eventually generate the lexicographically least squarefree infinite word over $\Sigma_3$, surprisingly little is known about it. For example, we do not even know whether the number of positions that one has to backtrack is bounded.

This suggests dropping the backtracking entirely, by enlarging our alphabet to the set of natural numbers $\N$. (For some recent papers on words and morphisms over an infinite alphabet, see \cite{Ferenczi:2006,Gonidec:2006,Mauduit:2006}.) In this situation, the concept of irreducibility of words and morphisms, introduced in Section~\ref{irred}, becomes relevant. As we will see in Section~\ref{squares}, the resulting squarefree word,
\[\w_2 = 01020103010201040102010301020105\cdots,\]
is a famous one; it is the so-called ``ruler'' sequence, where the $n$th term is the exponent of the highest power of 2 dividing $n$, and it can be generated by iterating an irreducible squarefree morphism.

Instead of avoiding squares over $\N$, we could try to avoid overlaps. Using a greedy algorithm without backtracking, we generate the word
\[\w_{2^+} = {\scriptstyle 0010011001002001001100100210010020010011001002001001200100110010020010011001003}\cdots,\]
with many remarkable properties. 
Among other things, $\w_{2^+}$ is generated by iterating a certain irreducible overlap-free morphism, but in this case, the morphism is much more complicated. This is discussed in Sections~\ref{overlaps} and \ref{more}.

\section{Notation}

Our notation is mostly standard, but we collect it here for ease of reference.

An \emph{alphabet} $\Sigma$ is a set of symbols, called letters. Although alphabets are usually finite in the literature on combinatorics on words, in this paper we also consider the alphabet $\N = \{0, 1, 2, \ldots\}$ of natural numbers.

A \emph{word} over this alphabet is a (possibly empty) string of letters chosen from $\Sigma$. The empty word is denoted $\epsilon$, and the length of a word $w$ is denoted $|w|$. We write $w[n]$ for the $n$th letter of $w$ (with indexing starting at 1).

The set of all finite words over $\Sigma$ is denoted by $\Sigma^*$, the set of non-empty finite words by $\Sigma^+$, and the set of one-way right-infinite words by $\Sigma^\omega$.

The basic operation on words is concatenation. Usually we represent concatenation by juxtaposition, so that $x$ concatenated with $y$ is written $xy$. However, we sometimes write it as $x \cdot y$ for clarity; for example, $(n+1) \cdot (n+2)$ denotes the word of length 2 consisting of the letter $n+1$ followed by the letter $n+2$.

A word $y$ is a \emph{factor} of a word $w$ if there exist words $x, z$ such that $w = xyz$. If $x = \epsilon$, then $y$ is a \emph{prefix} of $w$; if $z = \epsilon$, then $y$ is a \emph{suffix} of $w$. If $y$ is a prefix (resp., suffix) of $w$, then we write $y^{-1}w$ (resp., $wy^{-1}$) to denote the word obtained by removing the prefix (resp., suffix) $y$ from $w$.

Given an ordering on the elements of $\Sigma$, there is an associated \emph{lexicographic order} on $\Sigma^* \cup \Sigma^\omega$. We write $x \leq y$ if $x$ is a prefix of $y$, or if we can write $x = wcx'$ and $y = wdy'$, where $w$ is a common prefix of $x$ and $y$ and $c, d$ are letters with $c < d$.

Given a set $P$ of words, called a \emph{pattern}, we say that $w$ \emph{avoids} $P$ (or that $w$ is \emph{$P$-free}) if no word of $P$ is a factor of $w$. Some examples of interesting patterns include the squares $\{xx : x \in \Sigma^+\}$, the cubes $\{xxx : x \in \Sigma^+\}$, and the overlaps $\{cxcxc : c \in \Sigma, x \in \Sigma^*\}$.

Let $\Sigma, \Delta$ be alphabets. A \emph{morphism} is a function $h \colon \Sigma^* \to \Delta^*$ such that $h(xy) = h(x)h(y)$ for all $x, y \in \Sigma^*$. To define a morphism, it suffices to give $h(c)$ for all letters $c \in \Sigma$.

The basic operation on morphisms is composition. If $\Sigma, \Delta, \Gamma$ are alphabets and $h \colon \Sigma^* \to \Delta^*$, $g \colon \Delta^* \to \Gamma^*$ are morphisms, then their composition $g \circ h \colon \Sigma^* \to \Gamma^*$ is also a morphism. If $\Sigma = \Delta$, so that $h \colon \Sigma^* \to \Sigma^*$, we can iterate it. We write $h^n$ for the $n$-fold composition of $h$ with itself, and let $h^0$ denote the identity map.

If $c \in \Sigma$ is a letter and $h \colon \Sigma^* \to \Sigma^*$ is a morphism with $h(c) = cx$ for some word $x$, then
\[h^n(c) = c \cdot x \cdot h(x) \cdot h^2(x) \cdot \cdots \cdot h^{n-1}(x).\]
If $h^n(x) \neq \epsilon$ for all $n \geq 0$, then there is a unique infinite word of which $c, h(c), h^2(c), \ldots$ are all prefixes, and we write it as $h^\omega(c)$.

Given a property of words, we say that a morphism has that property if it preserves the property when applied to words. For example, given a pattern $P$, we say that the morphism $h$ is $P$-free if $h(w)$ is $P$-free whenever $w$ is.

Given an alphabet $\Sigma$, we let $S \colon \Sigma^+ \to \Sigma^+$ be the left cyclic shift operator, defined by $S(cx) = xc$ for all $c \in \Sigma$ and $x \in \Sigma^*$, and we let $R \colon \Sigma^* \to \Sigma^*$ be the reversal operator, defined by $R(c) = c$ for $c \in \Sigma$ and $R(xy) = R(y)R(x)$ for $x,y \in \Sigma^*$. Note that these operators are not morphisms.

\section{Backtracking and no-backtracking algorithms}

As we noted, given a pattern $P$, we can ask whether there are infinite words avoiding $P$. For a finite alphabet $\Sigma$, this turns out to be equivalent to the existence of arbitrarily long finite words avoiding $P$, as the following algorithm shows:

\begin{enumerate}
	\item
		Start with the empty word $w_0 = \epsilon$.
	
	\item
		For each $i = 0, 1, 2, \ldots$, let $c_i \in \Sigma$ be a letter such that there are arbitrarily long words avoiding $P$ with $w_i c_i$ as a prefix, and set $w_{i+1} = w_i c_i$. Note that the existence of such a $c_i$ is guaranteed by the pigeonhole principle.
		
	\item
		Since $w_i$ is a prefix of $w_j$ whenever $i \leq j$, we can take $w = \lim_{i\to\infty} w_i$. Then $w$ is an infinite word over $\Sigma$ avoiding $P$.
\end{enumerate}

If we put an ordering on the letters of $\Sigma$ and choose $c_i$ to be minimal at each step, then the algorithm actually shows something slightly stronger: if there are arbitrarily long words avoiding $P$, then there is a \emph{lexicographically least} infinite word $\alpha$ avoiding $P$. Since it is not clear \textit{a priori} how to choose $c_i$ in step 2, we also have the following, more explicit algorithm which will either show that there are no infinite words avoiding $P$, or converge to $\alpha$:

\begin{enumerate}
	\item
		Start with the empty word $w = \epsilon$. Let $a$ and $z$ be the lexicographically smallest and largest letters in $\Sigma$, respectively.

	\item
		Repeat this step as long as possible: if $w$ does not have a suffix in $P$, append $a$ to it. Otherwise, remove all trailing $z$'s from $w$, and replace the last letter of $w$ by the lexicographically next one in $\Sigma$. This will fail if $w$ contains only $z$'s.

	\item
		If the preceding step ever fails, conclude that there is no infinite word avoiding $P$. Otherwise, $w$ will eventually start with longer and longer prefixes of $\alpha$.
\end{enumerate}

Unfortunately, while the algorithm converges to $\alpha$ (if it exists), it can be hard to determine if a given letter of $w$ is there to stay, or if it will eventually be replaced. One way around this difficulty is to consider patterns where no backtracking actually occurs in the second algorithm. In such cases, we get the \emph{no-backtracking} algorithm:

\begin{enumerate}
	\item
		Start with the empty word $w_0 = \epsilon$.

	\item
		For each $i = 0, 1, 2, \ldots$, let $c_i \in \Sigma$ be the lexicographically first letter in $\Sigma$ such that $w_i c_i$ does not have a suffix in $P$, if it exists, and set $w_{i+1} = w_i c_i$.
		
	\item
		If the preceding step never fails, then the $w_i$ are the prefixes of $\alpha$.
\end{enumerate}

This is what we consider in this paper. For the patterns of squares and overlaps over $\N$, the no-backtracking algorithm works, and we construct the resulting words. The squarefree word is well-known, but the overlap-free word is not, and we explore its structure.

\section{Irreducibility of words and morphisms}
\label{irred}

In the context of the no-backtracking algorithm, the concept of irreducibility becomes relevant. Given a pattern $P$ over an ordered alphabet $\Sigma$, we say that a word $w$ is \emph{irreducible at position $p$} (with respect to $P$) if replacing $w[p]$ with any lexicographically smaller letter in $\Sigma$ creates a new word with a factor in $P$ ending at position $p$. (Note that we allow the possibility that $w$ itself already has a factor in $P$ ending at that position.) In particular, if $w[p]$ is the smallest letter of $\Sigma$, then $w$ is automatically irreducible at $p$.

If a word is irreducible at every position, we simply say that it is \emph{irreducible}. Sometimes we will speak of words $w$ that are \emph{irreducible after the first position}, meaning that $w$ is irreducible at positions $2, 3, \ldots, |w|$.

These concepts are related to the lexicographic ordering over $\Sigma$ in the following way. If $v$ is irreducible with respect to $P$ and $w$ is $P$-free, then either $w$ is a prefix of $v$, or $v \leq w$ lexicographically. This can be seen by considering a longest common prefix $x$ of $v$ and $w$. Either this is all of $w$, or all of $v$, or each word contains a letter following $x$, in which case the next letter of $v$ must be strictly smaller than the next letter of $w$. It follows from this that if an infinite word $w$ is $P$-free and irreducible, then it is the lexicographically least infinite $P$-free word, and the finite $P$-free and irreducible words are exactly the prefixes of $w$.

\section{A squarefree word without backtracking}
\label{squares}

As a warmup, let us consider the case of squarefree words. For the rest of this section, we consider the pattern $P = \{xx : x \in \N^+\}$ of squares. Any finite squarefree word $w$ over $\N$ can be extended to a longer squarefree word by appending a letter that does not appear in $w$, so it follows that the no-backtracking algorithm will work and generate the lexicographically least infinite squarefree word over $\N$,
\[\w_2 = 01020103010201040102010301020105\cdots.\]
This is the well-known ruler sequence, which is sequence A007814 in Sloane's \textit{Encyclopedia} \cite{Sloane}. For other mentions of the ruler sequence, see \cite[Example 8, p.\ 187]{Allouche&Shallit:1992} and \cite{Er:1985}.

\begin{theorem}
Let $\gamma : \N^* \to \N^*$ be the morphism defined by $\gamma(i) = 0 \cdot (i+1)$. Then $\w_2 = \gamma^\omega(0)$.
\end{theorem}

\begin{proof}
We prove the result by showing that the morphism $\gamma$ is squarefree and irreducible.

Consider the morphism $\rho$ defined by $\rho(0) = \epsilon$ and $\rho(i) = i-1$ for $i \geq 1$. Then it is easy to see that $\rho$ is a left inverse of $\gamma$, in the sense that $\rho(\gamma(w)) = w$ for all words $w$. Suppose $\gamma(w)$ contains a square $xx$. Then $x$ contains at least one nonzero letter, so $w = \rho(\gamma(w))$ contains the nonempty square $\rho(x)\rho(x)$. Hence $\gamma$ is a squarefree morphism.

Now consider the letter $d$ at position $p$ in $\gamma(w)$, and suppose we replace it by a letter $c < d$. If $p$ is odd, then $d = 0$, so this cannot be done and $\gamma(w)$ is irreducible at this position. If $p$ is even, then $\gamma(w)[p-1]$ is 0, so taking $c = 0$ creates the square $00$ ending at position $p$. On the other hand, taking $c > 0$ creates a word of the form $\gamma(w')$, where $w'$ is obtained from $w$ by replacing the letter $d-1$ at position $p/2$ by the smaller letter $c-1$. The word $\gamma(w')$ has a square ending at position $p$ if and only if $w'$ has a square ending at position $p/2$. Thus, if $w$ is irreducible, then $\gamma(w)$ is irreducible, so $\gamma$ is an irreducible morphism.

It now follows from the discussion in Section~\ref{irred} that $\gamma^\omega(0)$ is the lexicographically least infinite squarefree word over $\N$.
\end{proof}

The following summarizes some folklore results about the ruler sequence.

\begin{corollary}
Let $\w_2 = 01020103 \cdots$, and let $\gamma$ be the morphism defined above. Then
\begin{enumerate}[(a)]
	\item
		$|\gamma^i (j)| = 2^i$ for $i \geq 0$;
	\item
		$\gamma^i (j)$ starts with $0$ and ends with $i+j$ for $i \geq 1$;
	\item
		$\w[i] = \nu_2 (i)$, the exponent of the highest power of 2 dividing $i$;
	\item
		The least index $i$ such that $\w[i] = j$ is $i = 2^j$;
	\item
		The letter $j$ occurs in $\w_2$ with limiting frequency $2^{-i-1}$.
\end{enumerate}
\end{corollary}

\begin{proof}
Left to the reader.
\end{proof}

\section{An overlap-free word without backtracking}
\label{overlaps}

For the rest of this paper, we consider the pattern $P = \{cxcxc : c \in \N, x \in \N^*\}$ of overlaps. As with squares, any finite overlap-free word over $\N$ can be extended to a longer overlap-free word by appending a letter that does not appear in it, so the lexicographically least infinite overlap-free word $\w_{2^+}$ over $\N$ exists and can be generated by using the no-backtracking algorithm.

We will show that $\w_{2^+}$ can be written as $\varphi^\omega(0)$ for a certain remarkable morphism $\varphi \colon \N^* \to \N^*$ with
\begin{align*}
	\varphi(0) &= 001 \\
	\varphi(1) &= 1001002 \\
	\varphi(2) &= 200100110010020010011001003 \\
	&\hspace{1.2ex}\vdots
\end{align*}
To do this, we will first define $\varphi$, and then show that it is both overlap-free and irreducible.

One particularly useful definition of $\varphi \colon \N^* \to \N^*$ is
\[\varphi(h) = (S^{-1}(\varphi^h(00))) \cdot (h+1), \qquad h \in \N,\]
but to make sure this definition is not circular and prove properties of $\varphi$, we need to be more careful. We will define a sequence of morphisms $\varphi_h \colon \{0,\ldots,h\}^* \to \{0,\ldots,h+1\}^*$ that extend each other, and let $\varphi$ be their limit.

\begin{definition}
For all $h \in \N$, let $\varphi_h \colon \{0,\ldots,h\}^* \to \{0,\ldots,h+1\}^*$ be defined by $\varphi_h(h') = \varphi_{h'}(h')$ for $h' < h$ and by
\[\varphi_h(h) = (S^{-1} \circ \varphi_{h-1} \circ \cdots \circ \varphi_0(00)) \cdot (h+1).\]
Note that for $h = 0$, this definition gives $\varphi_0(0) = S^{-1}(00) \cdot 1 = 001$. Since $\varphi_h$ extends $\varphi_{h'}$ for $h' < h$, it is meaningful to define $\varphi \colon \N^* \to \N^*$ to be their common extension.
\end{definition}

\begin{lemma}\label{form}
For all $h \in \N$, $\varphi(h)$ starts with $h$ and ends with $h+1$. Furthermore, if $w \in \{0,\ldots,h\}^*$, then there are as many occurrences of $h+1$ in $\varphi(w)$ as there are occurrences of $h$ in $w$, and each one is preceded by a 0.
\end{lemma}

\begin{proof}
We proceed by induction on $h$ with a vacuous base case. For every letter $h' < h$ that appears in $w$, the corresponding factor of $\varphi(w)$ is $\varphi(h') = \varphi_{h'}(h') \in \{0,\ldots,h'+1\}^* \subseteq \{0,\ldots,h\}^*$, so $\varphi(h')$ does not contain an occurrence of $h+1$.

We also have
\[\varphi_h(h) = (S^{-1} \circ \varphi_{h-1} \circ \cdots \circ \varphi_0(00) ) \cdot (h+1) \in \{0,\ldots,h\}^*(h+1),\]
so for each occurrence of $h$ in $w$, the corresponding factor $\varphi(h) = \varphi_h(h)$ in $\varphi(w)$ contains exactly one occurrence of $h+1$. By induction, the last letter of $\varphi_{h-1} \circ \cdots \circ \varphi_0(00)$ is $h$, and it is preceded by 0, so $\varphi_h(h)$ starts with $h$ and ends with $0 \cdot (h+1)$. Since all occurrences of $h+1$ in $\varphi(w)$ occur in this way, this completes the proof.
\end{proof}

\begin{theorem}\label{olirred}
For all $h \in \N$, $\varphi_h$ is irreducible and irreducible after the first position with respect to overlaps. Thus, $\varphi = \lim_{h\to\infty} \varphi_h$ has these properties.
\end{theorem}

\begin{proof}
We proceed by induction on $h$ with a vacuous base case. 

First, let us show that for each $h' \leq h$, the word $\varphi_h(h')$ is irreducible after the first position. For $h' < h$, the string $h'$ is irreducible after the first position, so by induction, $\varphi_h(h') = \varphi_{h'}(h')$ is irreducible after the first position. For $h' = h$, the word $y = \varphi_{h-1} \circ \cdots \circ \varphi_0(00)$ is irreducible, a square, and ends with the letter $h$ by Lemma~\ref{form}. Thus, the word
\[hy = (S^{-1} \circ \varphi_{h-1} \circ \cdots \circ \varphi_0(00)) \cdot h\]
is irreducible after the first position and is an overlap, so the word
\[\varphi_h(h) = (S^{-1} \circ \varphi_{h-1} \circ \cdots \circ \varphi_0(00)) \cdot (h+1)\]
is irreducible after the first position. 

Now let $w \in \{0,\ldots,h\}^*$ be a word. Then $\varphi_h(w)$ can be broken up into blocks corresponding to the images under $\varphi_h$ of the individual letters in $w$. By the remarks above, each position in $\varphi_h(w)$ that is not the first position of a block is irreducible.

By Lemma~\ref{form}, we can recover $w$ from $\varphi_h(w) = \varphi(w)$ by taking the letters in the first position of each block. Suppose we replace the letter $d$ at one of these positions $p$ by a letter $c < d$. If this creates an overlap in $w$ ending at $p$, then position $p$ is preceded by a square $xx$ in $w$ that begins with $c$. This gives a square $\varphi_h(x)\varphi_h(x)$ in $\varphi_h(w)$ that begins with $c$, so replacing $d$ by $c$ in $\varphi_h(w)$ creates an overlap ending at that position. Thus, if $w$ is irreducible at a position, then $\varphi_h(w)$ is irreducible at the first position of the corresponding block. If $w$ is irreducible or irreducible after the first position, then $\varphi_h(w)$ has the same property, so $\varphi_h$ is both irreducible and irreducible after the first position.
\end{proof}

\begin{theorem}\label{olfree}
For all $h \in \N$, $\varphi_h$ is an overlap-free morphism. Thus, $\varphi = \lim_{h\to\infty} \varphi_h$ is an overlap-free morphism.
\end{theorem}

\begin{proof}
We proceed by induction on $h$ with a vacuous base case.

First, let us show that for each $h' \leq h$, the word $\varphi_h(h')$ is overlap-free. For $h' < h$, the string $h'$ is overlap-free, so by induction, $\varphi_h(h') = \varphi_{h'}(h')$ is overlap-free. For $h' = h$, the word 00 is overlap-free, so $y = \varphi_{h-1} \circ \cdots \circ \varphi_0(00)$ is overlap-free. By Lemma~\ref{form}, $y$ contains exactly two occurrences of the letter $h$, and no occurrences of the letter $h+1$. Thus,
\[hy = (S^{-1} \circ \varphi_{h-1} \circ \cdots \circ \varphi_0(00)) \cdot h\]
is itself an overlap, and it does not contain any other overlap, so
\[\varphi_h(h) = (S^{-1} \circ \varphi_{h-1} \circ \cdots \circ \varphi_0(00)) \cdot (h+1)\]
is overlap-free. 

Now let $w \in \{0,\ldots,h\}^*$ be a word, and break up $\varphi_h(w)$ into blocks corresponding to the images under $\varphi_h$ of the individual letters in $w$. Suppose $x$ is an overlap of length $2n+1$ in $\varphi_h(w)$, so that $x[i] = x[i+n]$ for all $1 \leq i \leq n+1$. We want to show that $w$ contains an overlap.
	
Since $x$ is not contained in a single block of $\varphi_h(w)$ by the remarks above, let $x[k]$ be the start of the block $B$ containing $x[2n+1]$, so that $1 < k \leq 2n+1$. Then $x[k-1]$ is the end of a block, so it is not 0. If $k \leq n+1$, then $x[k+n] = x[k]$ is contained in the block $B$, and it follows from Lemma~\ref{form} that $x[k-1] = x[k+n-1] = 0$, a contradiction. Thus, we actually have $n+1 < k \leq 2n+1$.

Since the block $B$ starting at $x[k]$ contains $x[2n+1]$, we have $x[k] \geq x[k], \ldots, x[2n]$. Consider the block $A$ that contains $x[k-n]$, and say $A$ starts at $x[j]$ and ends at $x[\ell]$. Since $x[k-n-1] = x[k-1] \neq 0$, the block $A$ does not end at $x[k-n]$, so $x[j] \geq x[k-n]$. Since $x[k-n] \geq x[k-n], \ldots, x[n]$, the block $A$ does not end at any of these positions, so $n+1 \leq \ell < k$. Since $x[\ell-n] = x[\ell] > x[j]$, the block $A$ starts after this position, so $1 \leq \ell-n < j \leq k-n \leq n+1$.

Since $x[j-1]$ is the end of a block, it is not 0. Thus, $x[j+n-1] \neq 0$, so $x[j+n]$ is not the end of a block. Since $x[j] \geq x[j], \ldots, x[\ell-1]$, we have $x[j+n] \geq x[j+n], \ldots, x[2n]$, so the block containing $x[j+n]$ does not end at any of these positions. Thus, the block containing $x[j+n]$ also contains $x[2n+1]$, so $j+n \geq k$. Since we have $j \leq k-n$ from above, we have $j = k-n$.
	
The picture so far is that $x[k]$ is the start of the block $B$ containing $x[2n+1]$, and that $x[k-n]$ is the start of the block $A$ that ends at $x[\ell]$, where $n+1 \leq \ell < k \leq 2n+1$. Let $sxt$ be the factor of $\varphi_h(w)$ formed by taking the blocks containing $x$. Then we have
\begin{align*}
	sxt
		&= s \cdot x[1 \ldots k-n-1] \cdot x[k-n \ldots \ell] \cdot x[\ell+1 \ldots k-1] \cdot x[k \ldots 2n+1] \cdot t \\
		&= s \cdot x[1 \ldots k-n-1] \cdot \varphi_h(x[k]) \cdot x[\ell+1 \ldots k-1] \cdot \varphi_h(x[k]) \\
		&= s \cdot x[1 \ldots \ell-n] \cdot \varphi_h(z) \cdot \varphi_h(x[k]) \cdot \varphi_h(z) \cdot \varphi_h(x[k]) \\
		&= \varphi_h(x[k]) \cdot \varphi_h(z) \cdot \varphi_h(x[k]) \cdot \varphi_h(z) \cdot \varphi_h(x[k]) \\
		&= \varphi_h(x[k] \cdot z \cdot x[k] \cdot z \cdot x[k]),
\end{align*}
where $z$ is a (possibly empty) string, and the next-to-last equality holds because the string $s \cdot x[1 \ldots \ell-n]$ is non-empty and ends with $x[\ell-n] = x[\ell] = x[k]+1$. By Lemma~\ref{form}, $\varphi_h(h')$ starts with $h'$ and ends with $h'+1$ for every letter $h' \leq h$, so the last letter of a block (or the first one) completely determines which letter it comes from, and it follows that $\varphi_h$ is injective. Thus, $w$ must actually contain the string $x[k] \cdot z \cdot x[k] \cdot z \cdot x[k]$, which is an overlap. This shows that $\varphi_h$ is overlap-free.
\end{proof}

\begin{remark}
Note that our proof of Theorem~\ref{olfree} only depends on three facts about $\varphi(h)$. For all $h \in \N$,
\begin{enumerate}
	\item $\varphi(h)$ is overlap-free;
	\item $\varphi(h) \in h\{0,\ldots,h\}^*(h+1)$;
	\item every occurrence of $h$ or $h+1$ in $\varphi(h)$ after the first letter is preceded by 0.
\end{enumerate}
Thus, we know that various other morphisms from $\N^* \to \N^*$, such as the morphism defined by $h \mapsto h \cdot 0 \cdot (h+1)$, are also overlap-free.
\end{remark}

\begin{corollary}
The word $\varphi^\omega(0)$ is the lexicographically least infinite overlap-free word over $\N$.
\end{corollary}

\begin{proof}
By Theorems \ref{olirred} and \ref{olfree}, the infinite word $\varphi^\omega(0)$ is overlap-free and irreducible, so by the remarks of Section~\ref{irred}, it is the lexicographically least infinite overlap-free word over $\N$.
\end{proof}

This shows our main result, that $\w_{2^+} = \varphi^\omega(0)$, but we also get the following interesting corollary, which is the starting point for our exploration of the structure of $\varphi$ in the next section.

\begin{corollary}\label{psicor}
For all $0 \leq h \leq k$, let $\psi(h,k)$ be the lexicographically least overlap-free word over $\N$ that starts with $h$ and ends with $k$. Then $\psi(h,k) = \varphi^{k-h}(h)$.
\end{corollary}

\begin{proof}
Let $w$ be any overlap-free word starting with $h$ and ending with $k$. The word $h$ is overlap-free and irreducible after the first position, so by Theorems \ref{olirred} and \ref{olfree} the word $\varphi^{k-h}(h)$ is overlap-free and irreducible after the first position. Also, by Lemma~\ref{form}, it starts with $h$ and contains a single occurrence of $k$, at the end. Since $w$ contains $k$, it cannot be a proper prefix of $\varphi^{k-h}(h)$. Since $\varphi^{k-h}(h)$ is irreducible after the first position and $w$ is overlap-free and starting with the same letter, we have $\varphi^{k-h}(h) \leq w$ lexicographically. Thus, $\psi(h,k) = \varphi^{k-h}(h)$.
\end{proof}

\section{More about the overlap-free words \texorpdfstring{$\protect\psi(h,k)$}{psi(h,k)}}
\label{more}

For $0 \leq h \leq \ell \leq k$, the word $\psi(h,k)$ has $\psi(h,\ell)$ as a prefix and $\psi(\ell,k)$ as a suffix, since
\[\psi(h,k) = \varphi^{k-h}(h) = \varphi^{\ell-h}(\varphi^{k-\ell}(h)) = \varphi^{k-\ell}(\varphi^{\ell-h}(h)),\]
and $\varphi^{k-\ell}(h)$ starts with $h$, and $\varphi^{\ell-h}(h)$ ends with $\ell$.

However, we can be much more precise than this about the structure of $\psi(h,k)$. Letting $a(h,k) = |\psi(h,k)|$, we have the following result:

\begin{theorem}\label{fact}
Let $0 \leq h \leq k$. Then
\[\psi(h,k) = \left( \prod_{\ell=h}^{k-1} S^{-a(\ell,k-1)}(\psi(0,k-1))^2 \right) \cdot k.\]
\end{theorem}

\begin{proof}
The result is immediate for $h = k$, since then $\psi(k,k) = k$. For $h < k$ we have
\begin{align*}
	\psi(h,k)
		&= \varphi^{k-h-1}(\varphi(h)) \\
		&= \varphi^{k-h-1}(S^{-1}(\varphi^h(00)) \cdot (h+1)) \\
		&= S^{-|\varphi^{k-h-1}(h)|}(\varphi^{k-1}(00)) \cdot \varphi^{k-h-1}(h+1) \\
		&= S^{-|\psi(h,k-1)|}(\psi(0,k-1))^2 \cdot \psi(h+1,k),
\end{align*}
and the result follows by induction on $k-h$.
\end{proof}

\begin{corollary}\label{ahk}
We have the recurrence
\[a(h,k) = 2(k-h) a(0,k-1) + 1\]
for $0 \leq h \leq k$ and $k \geq 1$, with initial condition $a(0,0) = 1$. Furthermore,
\[a(0,k) = \sum_{\ell=0}^k \frac{2^k k!}{2^\ell \ell!} = \lfloor 2^k k! \sqrt{e} \rfloor.\]
\end{corollary}

\begin{proof}
Direct calculation.
\end{proof}

Table~\ref{tableahk} below gives the first few values of $a(h,k)$.

\begin{table}[H]
\begin{center}
\begin{tabular}{|c||r|r|r|r|r|r|}\hline
$h \backslash k$ & 0 & 1 & 2 & 3 & 4 & 5 \\ \hline \hline
0 & 1 & 3 & 13 & 79 & 633 & 6331 \\ \hline
1 &   & 1 &  7 & 53 & 475 & 5065 \\ \hline
2 &   &   &  1 & 27 & 317 & 3799 \\ \hline
3 &   &   &    &  1 & 159 & 2533 \\ \hline
4 &   &   &    &    &   1 & 1267 \\ \hline
5 &   &   &    &    &     &    1 \\ \hline
\end{tabular}
\end{center}
\caption{Some values of $a(h,k) = |\protect\psi(h,k)|$.}
\label{tableahk}
\end{table}

Note that the sequence $\big( a(0,k) \big)_k = (1, 3, 13, 79, 633, \ldots)$ is Sloane's sequence A010844 and has exponential generating function $\exp(x)/(1-2x)$. Another sequence of interest is $\big( |\varphi(h)| \big)_h = (3, 7, 27, 159, 1267, \ldots)$, which given by $|\varphi(h)| = a(h,h+1) = 2a(0,h) + 1$.

Note also that given the structure from Theorem~\ref{fact} and the fact that the function $a(0,k)$ grows quite fast and can be computed easily, it is possible to compute $\w_{2^+}[n]$ in time bounded by a polynomial in $\log n$, using the following algorithm:

\begin{tabbing}
while \= while \= while \= while \= \kill

function eval($n$): \\

\> \textit{// Find the first value of $a(0,k)$ which is at least $n$.} \\
\> $k := 0$; \\
\> $a[0] := 1$; \\
\> while ($a[k] < n$) do \\
\>\> $k := k+1$; \\
\>\> $a[k] := 2ka[k-1] + 1$; \\

\> \textit{// Compute $\psi(0,k)[n]$. This quantity is the loop invariant.} \\
\> while ($k \geq 0$) do \\
\>\> if ($n = a[k]$) then \\
\>\>\> \textit{// The last letter of $\psi(0,k)$ is $k$.} \\
\>\>\> return($k$); \\
\>\> else \\
\>\>\> \textit{// The letter falls in a block of the form $S^{-a(\ell,k-1)}(\psi(0,k-1))^2$.} \\
\>\>\> if ($k-1 > 0$) then \\
\>\>\>\> $\ell := \lfloor (n-1) / 2a[k-1] \rfloor$; \\
\>\>\>\> $\mathrm{shift} := 2 \ell a[k-2]$; \\
\>\>\> else \\
\>\>\>\> $\mathrm{shift} := 0$; \\
\>\>\> \textit{// Maintain the loop invariant while reducing $k$.} \\
\>\>\> $n := ((n + \mathrm{shift} - 1) \bmod a[k-1]) + 1$; \\
\>\>\> $k := k-1$;
\end{tabbing}

The next thing to consider is the frequency of each letter in $\w_{2^+}$ and the words $\psi(h,k)$. For fixed points of morphisms over a finite alphabet generated by iteration, it is well-known that the frequency of a letter, if it is exists, must be an algebraic number \cite[Thm.\ 8.4.5]{Allouche&Shallit:2003}. Now $\w_{2^+}$ is the fixed point of a morphism, but over an infinite alphabet. As we will see the frequency of each letter is transcendental.

The following corollary gives the distribution of letters for $\psi(0,k)$, from which the distribution of letters for $\psi(h,k)$ can easily be computed.

\begin{corollary}
Let $d(h,k)$ be the number of times $h$ occurs in $\psi(0,k)$. Then we have the recurrence
\[d(h,k) = 2k d(h,k-1)\]
for $0 \leq h < k$, with initial conditions $d(k,k) = 1$ for $k \geq 0$. Hence $d(h,k) = 2^{k-h} k!/h!$ for $0 \leq h \leq k$.
\end{corollary}

\begin{proof}
The recurrence follows directly from Theorem~\ref{fact}, and the rest is a direct calculation.
\end{proof}

Table~\ref{tabledhk} below gives the first few values of $d(h,k)$.

\begin{table}[H]
\begin{center}
\begin{tabular}{|c||r|r|r|r|r|r|}\hline
$h \backslash k$ & 0 & 1 & 2 & 3 & 4 & 5 \\ \hline \hline
0 & 1 & 2 & 8 & 48 & 384 & 3840 \\ \hline
1 &   & 1 & 4 & 24 & 192 & 1920 \\ \hline
2 &   &   & 1 &  6 &  48 &  480 \\ \hline
3 &   &   &   &  1 &   8 &   80 \\ \hline
4 &   &   &   &    &   1 &   10 \\ \hline
5 &   &   &   &    &     &    1 \\ \hline
\end{tabular}
\end{center}
\caption{Some values of $d(h,k)$, giving the letter frequencies in $\protect\psi(0,k)$.}
\label{tabledhk}
\end{table}

\begin{theorem}
For all $k \in \N$, the limiting frequency of the letter $k$ in $\w_{2^+}$ exists and is equal to $1/2^k k! \sqrt{e}$.
\end{theorem}

\begin{proof}
First we establish the relative frequencies of the letters. For all letters $k \in \N$ and lengths $n \geq 1$, let $f(k,n)$ be the number of occurrences of the letter $k$ in the prefix of $\w_{2^+}$ of length $n$. From Theorem~\ref{fact}, we know that rotations of $\psi(0,k)$ of the form $S^{-a(h,k)}(\psi(0,k))$ with $0 \leq h \leq k$ can be decomposed as concatenations of the letter $k$ and of rotations of $\psi(0,k-1)$ of the form $S^{-a(h',k-1)}(\psi(0,k-1))$ with $0 \leq h' \leq k-1$, in some order. By induction, it follows that for all $k' \geq k$, $\psi(0,k')$ can be decomposed as a concatenation of letters greater than $k$ and of rotations of $\psi(0,k)$ in some order.
	
Thus, each prefix of $\w_{2^+}$ consists of letters greater than $k$, a certain number of rotations of $\psi(0,k)$, and possibly a prefix of a rotation of $\psi(0,k)$. Since each rotation of $\psi(0,k)$ contains a single occurrence of the letter $k$ and exactly $2k$ occurrences of the letter $k-1$, we have $\lim_{n\to\infty} f(k-1,n)/f(k,n) = 2k$ and
\[\lim_{n\to\infty} \frac{f(k,n)}{f(0,n)} = \frac{1}{2^k k!}.\]

Next we show that $\lim_{n\to\infty} n/f(0,n) = \sqrt{e}$. Letting $F(\ell,n) = \sum_{k=0}^k f(k,n)$, we have
\[\lim_{n\to\infty} \frac{F(\ell,n)}{f(0,n)} = \sum_{k=0}^\ell \frac{1}{2^k k!}.\]
Since $\lim_{\ell\to\infty} F(\ell,n) = n$ pointwise, it is tempting to conclude that
\[\lim_{n\to\infty} \frac{n}{f(0,n)} = \sum_{k=0}^\infty \frac{1}{2^k k!} = \sqrt{e},\]
but for that we need some kind of uniform convergence, which we establish below.

From the remarks above on the decomposition of $\psi(0,k')$, it follows that for all $k,n$,
\[f(k,n) \leq \left\lceil \frac{n}{a(0,k)} \right\rceil.\]
Also, since the letter $k$ first appears at position $a(0,k)$, we have the more convenient bound
\[f(k,n) \leq \frac{2n}{a(0,k)} \leq \frac{2n}{2^k k!},\]
which is stronger for small values of $n$. Since $\sum_{k=0}^\infty 2/2^k k!$ is a convergent series, we get
\[\lim_{\ell\to\infty} \frac{F(\ell,n)}{n} = 1 - \lim_{\ell\to\infty} \sum_{k=\ell+1}^\infty \frac{f(k,n)}{n} = 1\]
uniformly in $n$.

Since the convergence is uniform in $n$ and the relative letter frequencies converge, writing
\[\frac{n}{f(0,n)} = \frac{F(\ell,n)}{f(0,n)} \cdot \frac{n}{F(\ell,n)}\]
for large enough $\ell$ shows that the quantity $n/f(0,n)$ can be bounded independently of $n$. Then, we have
\[\lim_{\ell\to\infty} \frac{F(\ell,n)}{f(0,n)} = \lim_{\ell\to\infty} \frac{n}{f(0,n)} \cdot \frac{F(\ell,n)}{n} = \frac{n}{f(0,n)}\]
uniformly in $n$, so we get
\[\lim_{n\to\infty} \frac{n}{f(0,n)} = \lim_{\ell\to\infty} \lim_{n\to\infty} \frac{F(\ell,n)}{f(0,n)} = \sqrt{e}\]
as desired, and the result follows.
\end{proof}

Finally, we can also derive some symmetry properties of $\psi(h,k)$. A surprising number of them follow from the next result.

\begin{theorem}\label{mgp}
For $k \in \N$, we have $S^{-1}(\psi(0,k)) = R(\psi(0,k))$.
\end{theorem}

\begin{proof}
We proceed by induction on $k$ with a vacuous base case. For $0 \leq h \leq k-1$, we have that $a(0,k-1) - a(h,k-1) = a(k-1-h,k-1) - 1$, so
\begin{align*}
	S^{-1}(\psi(0,k))
		&= k \cdot \prod_{h=0}^{k-1} S^{-a(h,k-1)}(\psi(0,k-1))^2 \\
		&= k \cdot \prod_{h=0}^{k-1} S^{a(k-1-h,k-1)}(S^{-1}(\psi(0,k-1)))^2 \\
		&= k \cdot \prod_{h=0}^{k-1} S^{a(k-1-h,k-1)}(R(\psi(0,k-1)))^2 \\
		&= k \cdot \prod_{h=0}^{k-1} R(S^{-a(k-1-h,k-1)}(\psi(0,k-1)))^2 \\
		&= k \cdot R\left( \prod_{h'=0}^{k-1} S^{-a(h',k-1)}(\psi(0,k-1))^2 \right) \\
		&= R(\psi(0,k)). \qedhere
\end{align*}
\end{proof}

\begin{corollary}\label{hkc}\ %
\begin{enumerate}[(a)]
	\item
		For all $0 \leq h \leq k$, $R(\psi(k-h,k)) \cdot \psi(h,k) = k \cdot \psi(0,k)$;
	\item
		For all $k \geq 0$, the word $\psi(0,k) \cdot k^{-1}$ is a palindrome;
	\item
		For all $h \in \N$, $\varphi(h) = S^{-1}(\psi(0,h))^2 \cdot (h+1) = R(\psi(0,h))^2 \cdot (h+1)$;
	\item
		For all $h \in \N$, the word $h^{-1} \cdot \varphi(h) \cdot (h+1)^{-1}$ is a palindrome.
	\item
		For all $h \in \N$, the word $\varphi \circ R \circ \varphi(h)$ is a palindrome.
\end{enumerate}
\end{corollary}

\begin{proof}\ %
\begin{enumerate}[(a)]
	\item
		By Theorem~\ref{mgp}, we have
		\[k \cdot \psi(0,k) = S^{-1}(\psi(0,k)) \cdot k = R(\psi(0,k)) \cdot k.\]
		We know that $\psi(h,k)$ is a suffix of $\psi(0,k)$ of length $a(h,k)$ and that $R(\psi(k-h,k))$ is a prefix of $R(\psi(0,k))$ of length $a(k-h,k)$. By Corollary~\ref{ahk}, $a(k-h,k) + a(h,k) = 1 + a(0,k)$, so this prefix and this suffix actually form all of $k \cdot \psi(0,k)$ together.
	\item
		From part (a), we know that $k \cdot \psi(0,k) = R(\psi(0,k)) \cdot k = R(k \cdot \psi(0,k))$ is a palindrome, and removing the first and last letter gives the palindrome $\psi(0,k) \cdot k^{-1}$.
	\item
		From Corollary~\ref{psicor} and Theorem~\ref{mgp}, we get
		\begin{align*}
			\varphi(h)
				&= S^{-1}(\varphi^h(0))^2 \cdot (h+1) \\
				&= S^{-1}(\psi(0,h))^2 \cdot (h+1) \\
				&= R(\psi(0,h))^2 \cdot (h+1)
		\end{align*}
	\item
		From part (c), we have
		\[\varphi(h) = R(\psi(0,h))^2 \cdot (h+1) = h \cdot R(\psi(0,h) \cdot h^{-1})^2 \cdot (h+1),\]
		and from part (b), $\psi(0,h) \cdot h^{-1}$ is a palindrome.
	\item
		We have
		\begin{align*}
			\varphi \circ R \circ \varphi(h)
				&= \varphi \circ R(R(\varphi^h(00)) \cdot (h+1)) \\
				&= \varphi((h+1) \cdot \varphi^h(00)) \\
				&= R(\varphi^{h+1}(00)) \cdot (h+2) \cdot \varphi^{h+1}(00),
		\end{align*}
		which is a palindrome. \qedhere
\end{enumerate}
\end{proof}

Given a morphism $\xi$, we can define its reversed morphism $\xi_R$ by reversing the image of each letter, so that $\xi_R = R \circ \xi \circ R$. Some rare morphisms, such as the Thue-Morse morphism $\mu \colon \{0,1\}^2 \to \{0,1\}^2$ defined by $\mu(0) = 01$ and $\mu(1) = 10$, have the property that they commute with their reversed morphism. The morphism $\varphi$ also has this property.

\begin{corollary}
The morphisms $\varphi$ and $\varphi_R$ commute.
\end{corollary}

\begin{proof}
It is enough to check that $\varphi \circ \varphi_R(h) = \varphi_R \circ \varphi(h)$ for all $h \in \N$. By Corollary~\ref{hkc}, we have
\[\varphi \circ \varphi_R(h) = \varphi \circ R \circ \varphi \circ R(h) = \varphi \circ R \circ \varphi(h) = R \circ \varphi \circ R \circ \varphi(h) = \varphi_R \circ \varphi(h). \qedhere\]
\end{proof}

We have already established the link between $\varphi$ and $\psi(h,k)$, but there is also a link between $\varphi_R$ and $\psi(h,k)$.

\begin{theorem}
For all $0 \leq h \leq k$ and $i \geq 0$, we have
\begin{align*}
	\psi(h,k+i) &= \varphi^i(\psi(h,k)) \\
	\psi(h+i,k+i) &= \varphi_R^i(\psi(h,k) \cdot k^{-1}) \cdot (k+i).
\end{align*}
\end{theorem}

\begin{proof}
The first equality follows directly from Corollary~\ref{psicor}. For the second equality, note that Corollary~\ref{hkc} gives
\[R(\psi(k-h,k)) \cdot \psi(h,k) = R(\psi(0,k)) \cdot k,\]
so that
\[\psi(h,k) \cdot k^{-1} = R(\psi(0,k) \cdot \psi(k-h,k)^{-1}).\]
Applying $\varphi_R^i = R \circ \varphi^i \circ R$ to both sides gives
\begin{align*}
	\varphi_R^i(\psi(h,k) \cdot k^{-1})
		&= R(\varphi^i(\psi(0,k) \cdot \psi(k-h,k)^{-1})) \\
		&= R(\psi(0,k+i) \cdot \psi(k-h,k+i)^{-1}) \\
		&= \psi(h+i,k+i) \cdot (k+i)^{-1}. \qedhere
\end{align*}
\end{proof}

\section{Further questions}

For the sake of simplicity, we have presented the proofs in this paper for squares and overlaps, but they can be extended easily enough to the case of arbitrary $n$th powers and $(n^+)$-powers for integer $n$, which is very similar. However, the case of fractional powers seems harder. (For a definition of fractional powers, see, for example \cite[p.\ 23]{Allouche&Shallit:2003}.) In fact, it is not even clear that an infinite alphabet is needed. For example, the first million letters of $\w_{5/2}$, the lexicographically least infinite word over $\N$ avoiding all powers with exponent $\geq 5/2$, are all in $\{0, 1, 2\}$.

Several other patterns $P$, especially when considered over $\N$, have the property that any finite $P$-free word can be extended to a longer word, in which case the no-backtracking algorithm will work. In such cases, the lexicographically least infinite $P$-free word is irreducible. One can ask, when is this word generated by a $P$-free irreducible morphism?

\end{document}